\documentclass[a4paper, 10pt, american, pdftex]{article}

\usepackage{babel}
\usepackage{color}
\usepackage{graphicx}
\usepackage{pinlabel}
\usepackage{hyperref}
\hypersetup{
  pdftitle={},
  pdfauthor={}, 
  pdfsubject={},
  pdfkeywords={}, 
  pdfpagemode=UseNone, 
  pdfstartview=Fit, 
  pdfnewwindow=true,
  breaklinks=true,
  colorlinks=true,
  linkcolor=black,
  anchorcolor=black,
  citecolor=black,
  filecolor=black,
  menucolor=black,
  urlcolor=black,
  bookmarksopen=false
}
\usepackage[charter, sfscaled=false]{mathdesign}
\usepackage{amsmath}
\usepackage{amsthm}
\usepackage[alwaysadjust]{paralist}
 \usepackage{pict2e}

\usepackage{pgf,tikz}
\usetikzlibrary{arrows}

\definecolor{qqqqff}{rgb}{0,0,1}
\definecolor{uququq}{rgb}{0.25,0.25,0.25}
\definecolor{cqcqcq}{rgb}{0.75,0.75,0.75}

\swapnumbers
\newtheorem{theorem}{Theorem}[section]
\newtheorem{corollary}[theorem]{Corollary}
\newtheorem{lemma}[theorem]{Lemma}
\theoremstyle{remark}
\theoremstyle{definition}
\newtheorem{definition}[theorem]{Definition}
\newtheorem{definitionandtheorem}[theorem]{Definition and Theorem}
\newtheorem{remark}[theorem]{Remark}

\newcommand{\R}{\mathbb{R}}
\newcommand{\C}{\mathbb{C}}

\newcommand{\RP}{\R\mathrm{P}}
\newcommand{\CP}{\C\mathrm{P}}
\newcommand{\GL}{\mathrm{GL}}
\newcommand{\SL}{\mathrm{SL}}

\date{} 
\author{Stefan Born, Ulrike B\"ucking, Boris Springborn}
\title{Quasiconformal dilatation of projective transformations and discrete conformal maps}

\begin{document}

\maketitle

\begin{abstract}
  We consider the quasiconformal
  dilatation of projective transformations of the real projective
  plane. For non-affine
  transformations, the contour lines of dilatation
  form a hyperbolic pencil of circles, and these are the only circles
  that are mapped to circles. We apply this result to analyze the
  dilatation of the circumcircle preserving piecewise
  projective interpolation between discretely conformally equivalent
  triangulations. We show that another interpolation scheme, angle
  bisector preserving piecewise projective interpolation, is in a
  sense optimal with respect to dilatation. These two
  interpolation schemes belong to a one-parameter family.

  \bigskip\noindent%
  30C62, 52C26
  
\end{abstract}

\section{Overview and motivation}
\label{sec:motivation}

The deviation of a function from being conformal is measured by its
quasiconformal dilatation. (Standard ways to quantify the local
deviation from conformality will be reviewed in
Section~\ref{sec:qc}.)  This paper is about the quasiconformal
dilatation of real projective transformations of the plane. In
Section~\ref{sec:circles1} we show that the contour lines of
dilatation form a hyperbolic pencil of circles
(Theorem~\ref{thm:hyppencil}). These are the only circles that are
mapped to circles (Theorem~\ref{thm:circles}). Although the statements
are elementary and the proofs are straightforward, these results seem
to be new. It may well be that the dilatation of a real projective map
was never considered before. We were motivated by the theory of
discrete conformal maps that is based on the following definition of
discrete conformal equivalence of triangle meshes. (Quasiconformal
dilatation also plays an important role in at least one other theory
of discrete conformality, circle packing~\cite{Stephenson05}.)

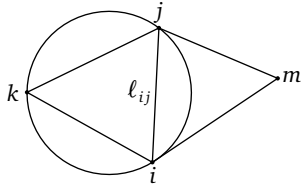
\begin{figure}[t]
\labellist
\small\hair 2pt
 \pinlabel {$i$} [t] at 48 5
 \pinlabel {$j$} [b] at 50 55
 \pinlabel {$k$} [r] at 0 31
 \pinlabel {$m$} [l] at 94 37
 \pinlabel {$\ell_{ij}$} [rm] at 43 31
\endlabellist
\centering
\includegraphics[scale=1.0]{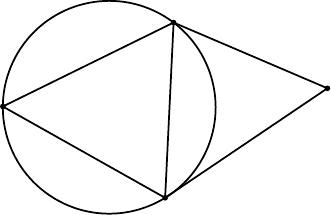}
\caption{Notation for the definition of discrete conformal
  equivalence. }
\label{fig:dconf}
\end{figure}

\begin{definitionandtheorem}[\cite{BPS10,Luo04}]
  \label{def:dconf}
  
  Two triangulated surfaces are considered \emph{discretely
    conformally equivalent}, if the triangulations are combinatorially
  equivalent and if one (and hence all) of the following equivalent
  conditions hold (see Figure~\ref{fig:dconf} for notation).

  \begin{compactenum}[(i)]
  \item The edge lengths $\ell_{ij}$ and $\tilde{\ell}_{ij}$ of
    corresponding edges are related by
    \begin{equation*}
      \tilde{\ell}_{ij}=e^{\frac{1}{2}(u_{i}+u_{j})}\ell_{ij}
    \end{equation*}
    for some logarithmic scale factors $u_{i}$ associated to the vertices.
  \item For interior edges $ij$, the \emph{length cross ratios} are equal,
    that is,
    \begin{equation*}
      \frac{\ell_{im}\ell_{jk}}{\ell_{mj}\ell_{ki}}
      =
      \frac{\tilde\ell_{im}\tilde\ell_{jk}}{\tilde\ell_{mj}\tilde\ell_{ki}} .
    \end{equation*}
  \item The \emph{circumcircle preserving projective maps} that map a
    triangle of one triangulation to the corresponding triangle of the
    other triangulation, and the respective circumcircles onto each
    other, fit together continuously across edges.
  \end{compactenum}
\end{definitionandtheorem}

The original definition~(i) is due to Feng Luo~\cite{Luo04}. For the
equivalent characterizations~(ii) and~(iii),
see~\cite{BPS10}. Characterization (iii) means that discretely
conformally equivalent triangle meshes allow not only the usual
piecewise linear (PL) interpolation (which works for any two
combinatorially equivalent triangulations) but also \emph{circumcircle
  preserving piecewise projective (CPP) interpolation} (which is not
continuous across edges unless two triangulations are discretely
conformally equivalent). This motivated the following
definition.

\begin{definition}[\cite{BPS10}]
  A simplicial continuous map between triangulated surfaces is a
  \emph{discrete conformal map} if the restriction to any triangle is
  a circumcircle preserving projective map onto the image triangle.
\end{definition}

Figure~\ref{fig:pl_vs_cpp}
\begin{figure}
  \centering
  
  \begin{minipage}[b]{158pt}
    \includegraphics[viewport=0 80 79 159, clip, scale=2]{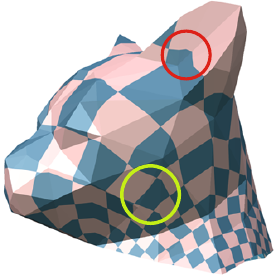}\\[-\baselineskip]
    \textsf{\small PL}
  \end{minipage}
  \quad
  \begin{minipage}[b]{158pt}
    \includegraphics[viewport=0 0 79 79, clip, scale=2]{figures/PLPJ}\\[-\baselineskip]
    \textsf{\small CPP}
  \end{minipage}\\[\bigskipamount]
  \begin{minipage}[b]{0.5\textwidth}
    \includegraphics[width=\textwidth]{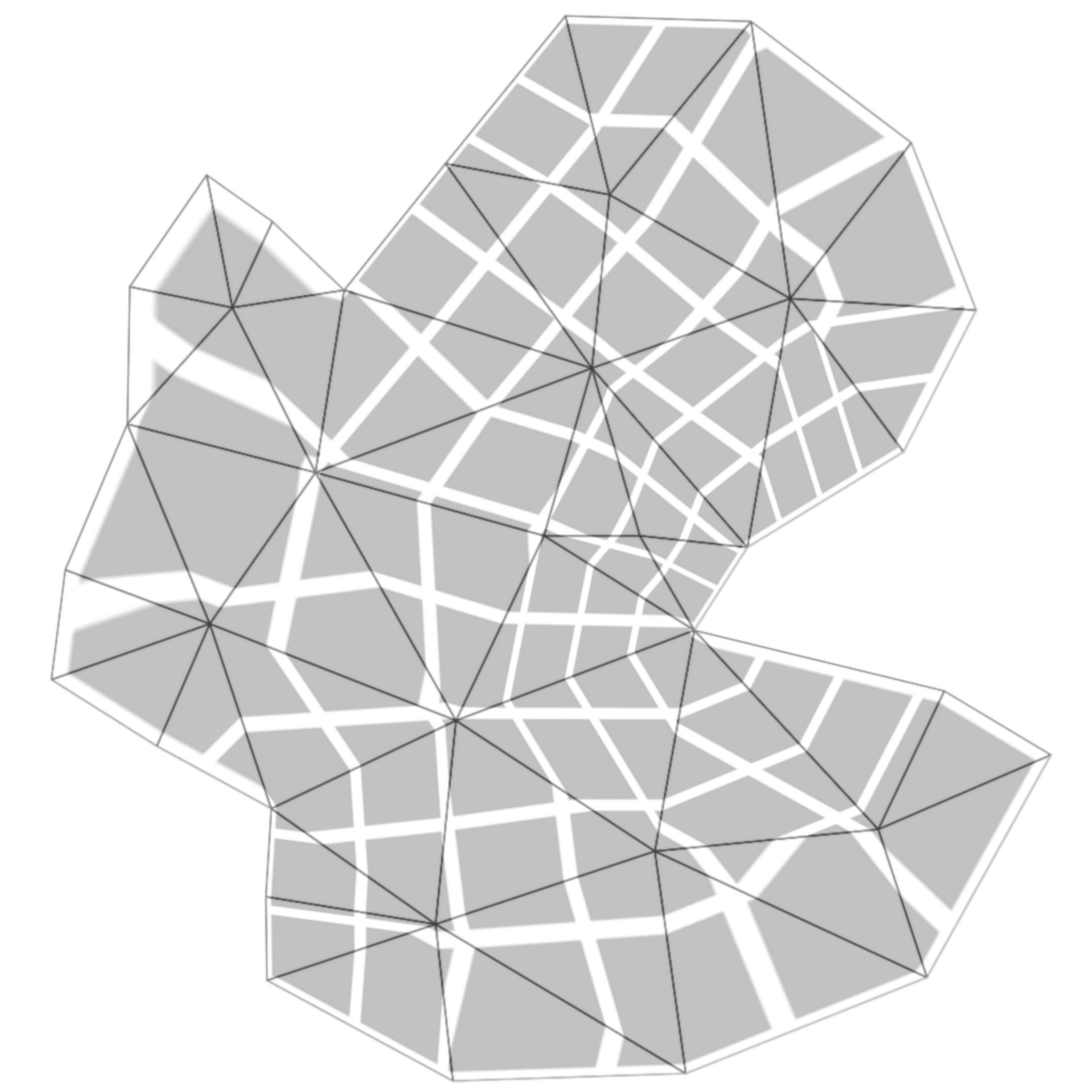}\\[-\baselineskip]
    \hspace*{0.5em}\textsf{\small PL}
  \end{minipage}%
  \begin{minipage}[b]{0.5\textwidth}
    \includegraphics[width=\textwidth]{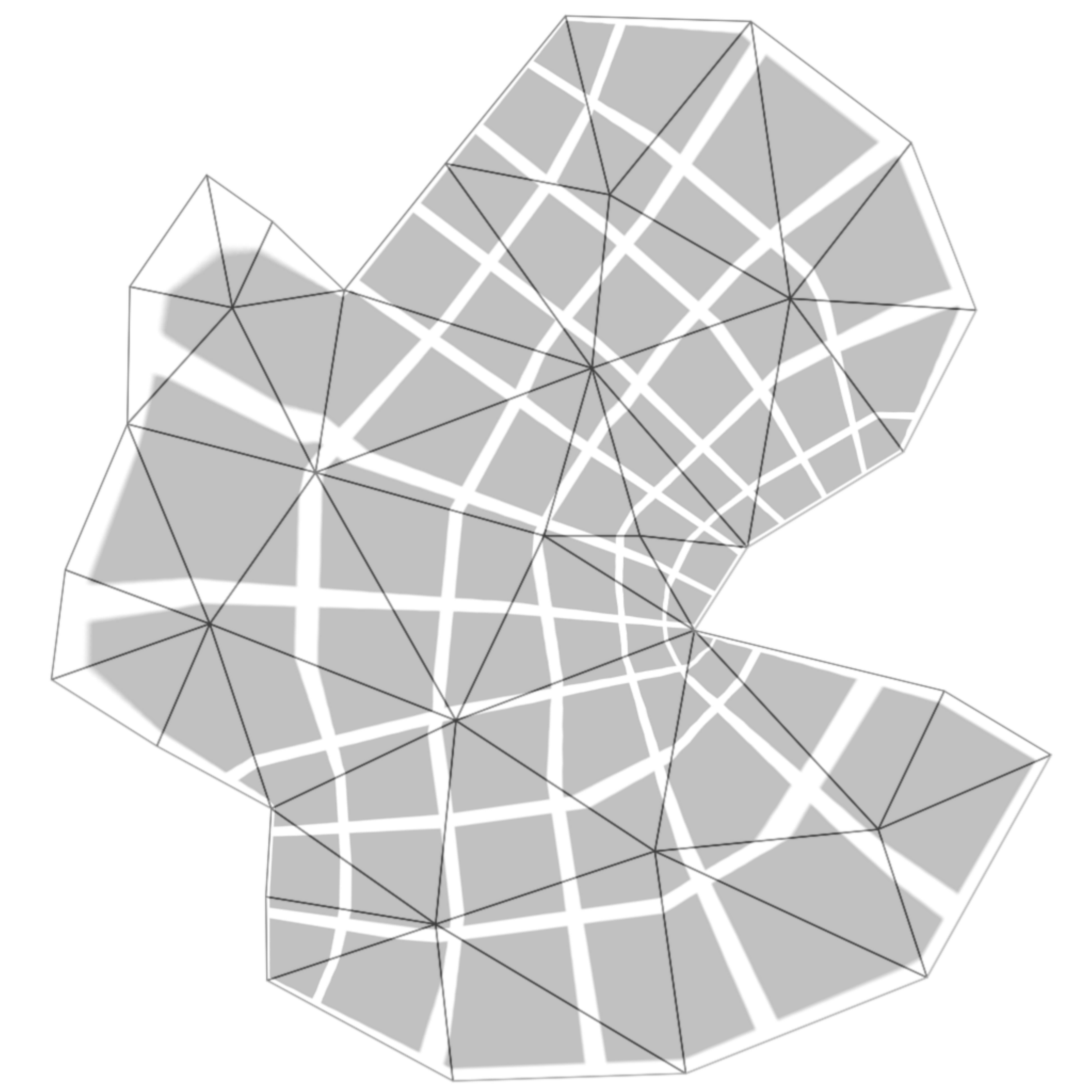}\\[-\baselineskip]
    \hspace*{0.5em}\textsf{\small CPP}
  \end{minipage}\\[\bigskipamount]
    \includegraphics[width=0.4\textwidth]{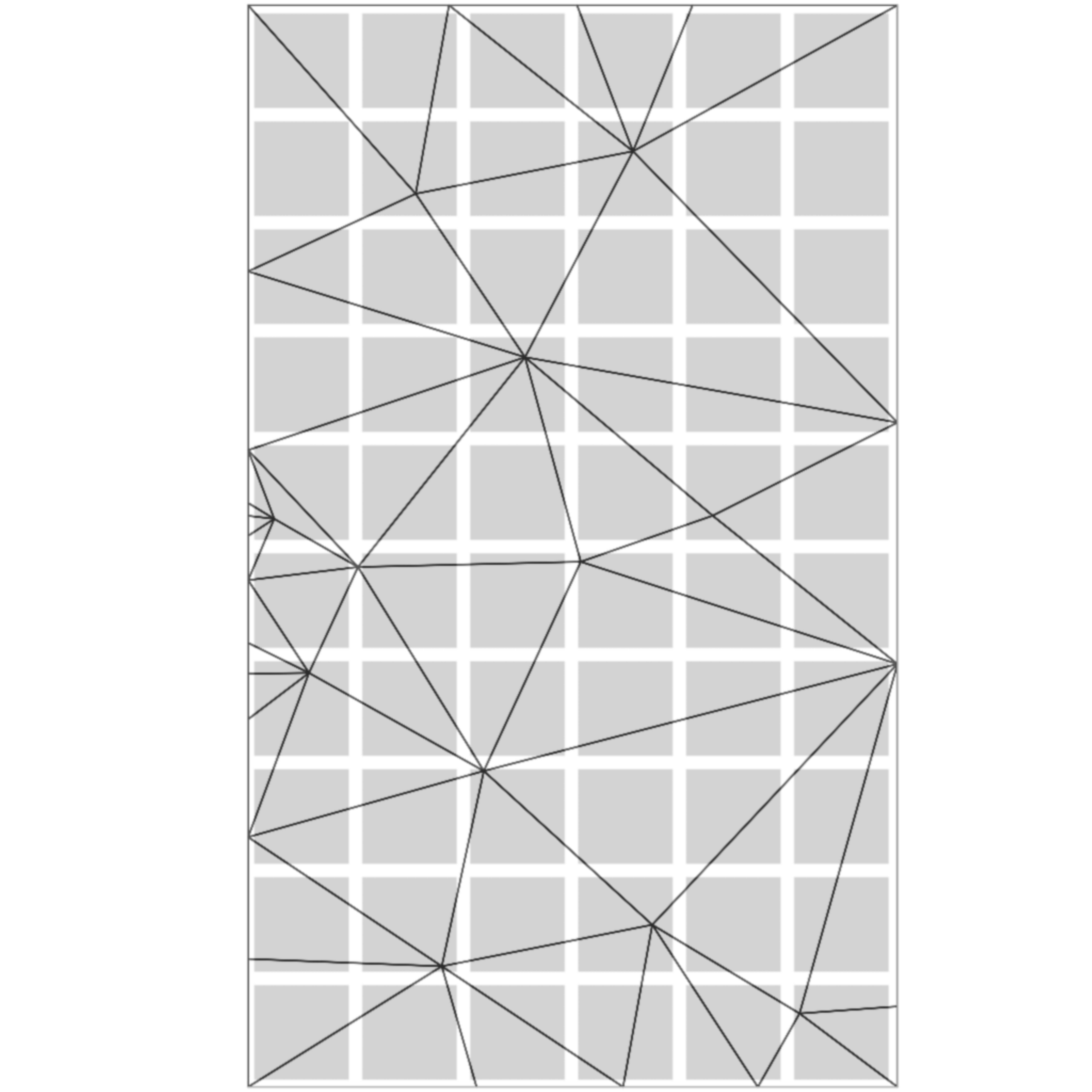}
  \caption{Piecewise linear (PL) vs circumcircle preserving piecewise
    projective (CPP) interpolation.  \emph{Top row~\cite{SSP08}:} A triangulated
    surface in $\R^{3}$ resembling a cat head is mapped to a
    conformally equivalent planar triangulation (not shown) by
    PL and CPP interpolation. The maps
    are visualized by pulling back a checkerboard pattern in the plane
    to the cat head. Some notable differences are highlighted.
    \emph{Middle row~\cite{BPS10}:} A triangulated planar region is mapped to a
    conformally equivalent triangulation of a rectangle~\emph{(bottom)}
    by PL and CPP interpolation.}
  \label{fig:pl_vs_cpp}
\end{figure}
shows visualizations of PL and CPP interpolations. The CPP
interpolations clearly look better. (This is an important point for
applications in computer graphics~\cite{SSP08}.) What is the reason?  This
question led us to investigate the quasiconformal dilatation of
projective transformations. We wanted to check the hypothesis that the
CPP interpolations had lower dilatation. Behind this
hypothesis was the non-mathematical hypothesis that lower
dilatation was the reason why CPP interpolation looks
better; however, see Remark~\ref{rem:why} in Section~\ref{sec:conclusion}.

As we show in Section~\ref{sec:cpp}, the dilatation of
CPP interpolation is indeed pointwise less than or equal to the
dilatation of PL interpolation, but the maximal dilatations per
triangle are equal (Corollary~\ref{cor:pointwise}). 

In Section~\ref{sec:angle_bisector}, we show that another
interpolation scheme, mapping angle bisectors to angle bisectors, is
in a sense optimal with respect to dilatation
(Theorem~\ref{thm:optimal}). Like CPP interpolation, angle bisector
preserving piecewise projective (APP) interpolation is continuous
across edges if and only if the triangulations are discretely
conformally equivalent (Theorem~\ref{thm:app_fit}).

The interpolations schemes PL, CPP, and APP are in fact members of a
continuous one-parameter family (see Section~\ref{sec:conclusion}). We do not
know any interesting geometric characterization for any other member
of this family.

\section{Dilatation and eccentricity}
\label{sec:qc}

In this section, we review the definitions of dilatation and eccentricity
of real differentiable maps. The definitions are standard (the
provided references are only meant as examples), but our
perspective is slightly unusual because (in
Section~\ref{sec:circles1}) we are interested in maps that become
singular and orientation reversing. (Real projective transformations
are generally not quasiconformal functions on $\C$.)

In the following, let $U \subseteq \C$ be open and let $f:U\to\C$ be a
real differentiable map. (We identify $\C$ and $\R^{2}$ as euclidean
vector spaces.) The function $f$ maps a small circle of radius
$\varepsilon$ around $z\in U$ approximately to an ellipse with major
semi-axis $\lambda_1\varepsilon$ and minor semi-axis
$\lambda_2\varepsilon$, where
\begin{equation*}
  \label{eq:lambda_ineq}
  \lambda_1\geq 
  \lambda_2\geq 0
\end{equation*}
are the singular values of the real derivative $df_{z}$ of $f$ at $z$.

The quotient of singular values is usually called the
\emph{dilatation}~\cite{Ahlfors_1954, Ahlfors_Lectures}, or, to be more specific,
the \emph{quasiconformal dilatation}~\cite{Papa_Theret}. However, since we
are interested in maps that may become singular and orientation
reversing, we define the \emph{[signed] dilatation}
$D_f(z)\in\R\cup\{\infty\}$ of $f$ at the point $z\in U$ by
\begin{equation*}
 D_f(z)= \pm \frac{\lambda_1}{\lambda_2},
\end{equation*}
where the sign is chosen according to whether $f$ is orientation
preserving ($+$) or reversing ($-$) at $z$. If the derivative $df$ is
singular at $z$, then $\lambda_{2}=0$ and we define $D_{f}(z)=\infty$. But we
will assume that $df$ vanishes nowhere, so the dilatation is well
defined everywhere.  Note the following properties of the dilatation:
\begin{compactitem}
\item $|D_f|\geq 1$,
\item $D_f=\pm 1$ where $f$ is conformal or 
  anticonformal, respectively,
\item $D_f(z)=D_{f^{-1}}(f(z))$.
\end{compactitem}

The 
\emph{Beltrami coefficient}
\begin{equation}
 \label{eq:defmu}
\mu_f(z)=f_{\bar{z}}/f_z
\end{equation}
is defined in terms of the Wirtinger derivatives
\begin{equation*}
 f_z=\frac{1}{2}(f_x-if_y),\qquad f_{\bar{z}}=\frac{1}{2}(f_x+if_y).
\end{equation*}
The modulus of the Beltrami coefficient is called the
\emph{eccentricity} of $f$~\cite{Ahlfors_1954}. 
The eccentricity $|\mu_f|$ is related to the dilatation $D_f$ by
\begin{equation*}
  |\mu_f| =\frac{D_f-1}{D_f+1}.
\end{equation*}
Note that $|\mu_f|\in[0,\infty]$, and
\begin{align*}
  |\mu_f| &=
            \begin{cases}
              0      & \text{ where $f_{\bar{z}}=0$, that is, where $f$ is 
conformal},\\ 
              \infty & \text{ where $f_{z}=0$, that is, where $f$ is 
anticonformal},
            \end{cases}\\
  |\mu_f| & \lesseqqgtr 1\quad\text{where}\quad\det df \gtreqqless 0,\\
\end{align*}
and
\begin{equation}
  \label{eq:mu_f_inv}
  |\mu_{f}(z)| = |\mu_{f^{-1}}(f(z))|.
\end{equation}

\begin{remark}[on terminology]
  (i) A function is called \emph{quasiconformal} if it is orientation
  preserving and its dilatation is bounded. In phrases like
  ``quasiconformal dilatation'', the adjective ``quasiconformal'' is
  used somewhat sloppily to mean ``belonging to the theory of
  quasiconformal functions''. (ii) The above definition of
  ``eccentricity'' has fallen into disuse. The
  Beltrami coefficient~$\mu_f$ is also called \emph{complex
    dilatation}~\cite{Ahlfors_Lectures}, although \emph{complex
    eccentricity} would make more sense.
\end{remark}

\section{Dilatation of a projective map, and circles mapped to circles}
\label{sec:circles1}

We are interested in the dilatation of a projective map
\begin{equation*}
  f:\RP^{2}\rightarrow\RP^{2},\quad [x]\mapsto[Ax],\quad A\in \GL_{3}(\R),
\end{equation*}
where we identify the complex plane $\C$ with the real projective plane
$\RP^{2}$ without the line $\ell_{\infty}=\{[x]\,|\,x_{3}=0\}$ via the
map
\begin{equation*}
  z=x+ iy
  \longleftrightarrow
  \Big[
    \begin{smallmatrix}
      x\\y\\1
    \end{smallmatrix}
    \Big].
\end{equation*}
This provides the conformal structure on
$\RP^{2}\setminus\ell_{\infty}$. We consider the line $\ell_{\infty}$
as the line at infinity. The map $f$ is [real] affine if it maps
$\ell_{\infty}$ to $\ell_{\infty}$. The map $f$ is a similarity
transformation if it is real affine and complex affine on $\C$, that is,
$z\mapsto az+b$ with $a\in\C\setminus\{0\}$, $b\in\C$.

The eccentricity of affine maps is constant, and
identically $0$ for similarity transformations. The following theorem
treats the interesting case (see Figure~\ref{fig:circles}).
\begin{figure}
  \centering
  \labellist
  \small\hair 2pt
  \pinlabel {$|\mu_{f}|=0$} [r] at 4 83
  \pinlabel {$|\mu_{f}|=\infty$} [l] at 235 83
  \pinlabel {$|\mu_{f}|=1$ on $f^{-1}(\ell_{\infty})$} [l] at 152 3
  \endlabellist
  \includegraphics[width=0.6\textwidth]{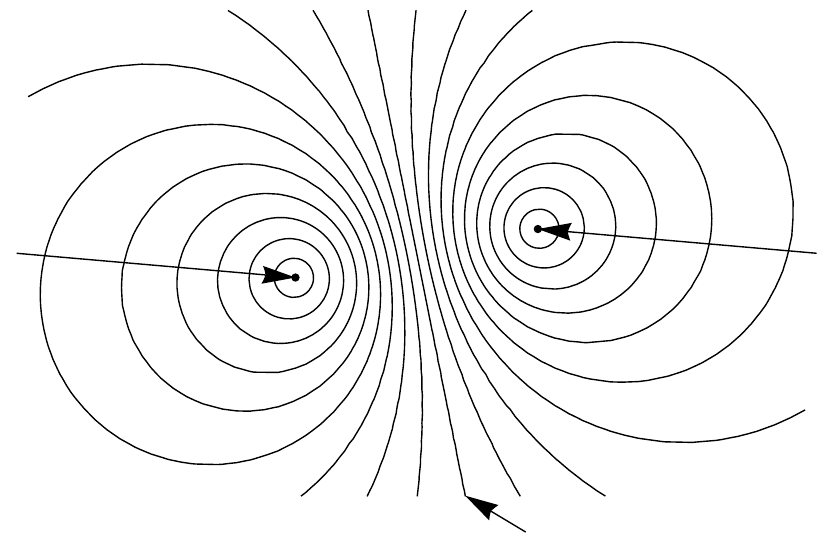}
  \caption{The contour lines of $|\mu_{f}|$ are a hyperbolic pencil of
    circles (Theorem~\ref{thm:hyppencil}).}
  \label{fig:circles}
\end{figure}

\begin{theorem}
\label{thm:hyppencil}
If the projective transformation $f:\RP^{2}\rightarrow\RP^{2}$ is not
affine then: 
\vspace{-0.5\baselineskip}
\begin{enumerate}[(i)]
\setlength\itemsep{-0.3em}
 \item The contour lines of $|\mu_{f}|$ are a hyperbolic pencil of circles.
\item The function~$|\mu_{f}|$ is an affine parameter for this
$1$-parameter family of circles that assigns the values~$0$
and~$\infty$ to the circles that degenerate to points and the
value~$1$ to the preimage of the line at infinity under $f$.
\item The function~$f$ maps this hyperbolic pencil of circles to
another hyperbolic pencil of circles
\end{enumerate}
\end{theorem}

The following proof relies on direct calculations. It would be nice to
have a more conceptual argument.

\begin{proof}
  In terms of the affine parameter $z$ on
  $\RP^{2}\setminus\ell_{\infty}$, the projective map $f$ may be written as
  \begin{equation}
    \label{eq:f_of_z}
    z \mapsto \frac{az+b\bar{z}+c}{pz+\bar{p}\bar{z}+q}, 
    \quad a,b,c,p\in\C,\quad q\in\R,
  \end{equation}
  where $p\neq 0$ because $f$ is not affine. In terms of the
  coefficients $a,b,c,p,q$, the matrix representing $f$ is
  \begin{equation*}
    \label{eq:A}
    A=
    \begin{pmatrix}
      a_{1}+b_{1} & -a_{2}+b_{2} & c_{1} \\
      a_{2}+b_{2} & a_{1}-b_{1} & c_{2} \\
      2p_{1} & -2p_{2} & q
    \end{pmatrix},
  \end{equation*}
  where indices $1$ and $2$ to denote real and imaginary parts, respectively.

  From the definition~\eqref{eq:defmu} of $\mu$, using the
  representation~\eqref{eq:f_of_z} for~$f$, one obtains immediately
  \begin{equation*}
    \label{eq:mu_f_projective}
    \mu_f(z)=\frac{\alpha z+\beta}{-\alpha \bar{z} +\gamma},
    \quad\text{where}\quad
    \left\{
      \begin{aligned}
        \alpha &= bp-a\bar{p}\\ 
        \beta  &= bq-c\bar{p}\\ 
        \gamma &= aq-cp
      \end{aligned}   
    \right.\ ,
  \end{equation*}
  so 
  \begin{equation*}
    |\mu_{f}(z)|=\left|\frac{\alpha z + \beta}{-\bar{\alpha}z+\bar{\gamma}}\right|.
  \end{equation*}
  As a tedious but straightforward calculation shows,
  \begin{equation*}
    \label{eq:det_A}
    \det
    \begin{pmatrix}
      \alpha & \beta \\ -\bar{\alpha} & \bar{\gamma}
    \end{pmatrix}
    =
    \bar{p}\det A \neq 0,
  \end{equation*}
  so $|\mu_{f}(z)|=|M(z)|$,
  where $M$ is a M\"obius transformation. Parts (i) and (ii) of the
  theorem follow easily.
 
  To see part~(iii), note that the inverse map $f^{-1}$ is again a
  projective map which is not affine. Therefore the contour lines of
  $|\mu_{f^{-1}}|$ are also the circles of a hyperbolic pencil of
  circles.  Because $f$ and $f^{-1}$ have the same dilatation at
  corresponding points (equation~\eqref{eq:mu_f_inv}), $f$ maps
  circles with constant $|\mu_{f}|$ to circles with constant
  $|\mu_{f^{-1}}|$. This proves~(iii).
\end{proof}

\begin{corollary}
  \label{cor:max_at_vertex}
  If $f$ is orientation preserving on a triangle $T$ then the function
  $|\mu_{f}|$ attains the maximum
  \begin{equation*}
    \max_{z\in T}|\mu_{f}(z)|
  \end{equation*}
  at a vertex.
\end{corollary}

Indeed, in the open half-plane where $f$ is orientation preserving, the
sublevel sets of $|\mu_{f}|$ are strictly convex (being disks), and
every point is on the boundary of its sublevel set. Hence, the maximum
cannot be attained in the interior of the triangle or in the relative
interior of a side.

\begin{remark}
  Here and throughout this article, ``triangle'' shall mean ``closed
  triangular region''.
\end{remark}

Theorem~\ref{thm:hyppencil} suggests the following question: Which
circles are mapped to circles by a projective map $f$? Obviously, a
similarity transformation maps all circle to circles, and an affine
transformation that is not a similarity maps all circle to ellipses
that are not circles. The following theorem treats the remaining
case.

\begin{theorem}
  \label{thm:circles}
  If the projective transformation $f:\RP^{2}\rightarrow\RP^{2}$ is
  not affine then it maps precisely the circles of a hyperbolic pencil
  to circles.
\end{theorem}

\begin{proof} 
  This proof relies on the classical characterization of circles in
  terms of their complex intersection points with the line at
  infinity. Consider the real projective plane $\RP^{2}$ as the set of
  points with real homogeneous coordinates in the complex projective
  plane $\CP^{2}$. In $\CP^{2}$, every conic section intersects the
  line at infinity twice (counting multiplicity). Circles are
  the conics intersecting the line at infinity in the \emph{imaginary
    circle points}
  \begin{equation*}
    \label{eq:circle_points}
    K= 
    \left[
      \begin{smallmatrix} 1\\i\\0 \end{smallmatrix} 
    \right],
    \quad
    \bar{K}= 
    \left[
      \begin{smallmatrix} 1\\-i\\0 \end{smallmatrix}
    \right].
  \end{equation*}

  The circles that $f$ maps to circles are the conics containing $K$,
  $\bar{K}$, $f^{-1}(K)$, and $f^{-1}(\bar{K})$. Because $f$ is real
  and not affine, these four points are in general position: No three
  of them are contained in a line, which would have to be the line at
  infinity or its preimage under $f$. But these lines intersect in a
  real point. Hence the conics containing the four points form a
  pencil of conics. Since the four points are two pairs of complex
  conjugates, the pencil is a hyperbolic pencil of real conics.
\end{proof}

\section{The circumcircle preserving projective map}
\label{sec:cpp}

The following theorem characterizes the maximal value of $|\mu_f|$ for a 
circumcircle preserving projective map on a triangle.

\begin{theorem}
\label{thm:cpp}
If the projective map $f$ maps a triangle $T$ with vertices $A$, $B$,
$C$ onto a triangle $\tilde T$ with vertices $\tilde A$, $\tilde B$,
$\tilde C$, and the circumcircle of $T$ to the circumcircle of
$\tilde T$, then
  \begin{equation}
    \label{eq:cpp_distortion}
    |\mu_{f}(A)|=|\mu_{f}(B)|=|\mu_{f}(C)|=|\mu_{h}|,
  \end{equation}
  where $h$ is the real affine map from $T$ onto $\tilde T$ (whose
  eccentricity $|\mu_{h}|$ is constant). 
\end{theorem}

Together with the results of the previous section, Theorem~\ref{thm:cpp}
implies the following corollary. 

\begin{corollary}
  \label{cor:pointwise}
  If a projective map $f$ satisfying the assumptions of
  Theorem~\ref{thm:cpp} is orientation preserving on $T$, then
  \begin{equation*}
    \label{eq:pointwise}
    |\mu_{f}(P)|\leq|\mu_{h}(P)|,
  \end{equation*}
  for all points $P\in T$, and
  \begin{equation*}
    \label{eq:maxcpp}
    \max_{P\in T}|\mu_{f}(P)|=|\mu_{h}|.
  \end{equation*}
\end{corollary}

\begin{remark}
  The projective map $f$ of Theorem~\ref{thm:cpp} is either orientation preserving or
  orientation reversing on all of~$T$. Since $f$ maps circumcircle to
  circumcircle, neither $\ell_{\infty}$ nor $f^{-1}(\ell_{\infty})$
  intersects~$T$.
\end{remark}

\begin{proof}[Proof of Theorem~\ref{thm:cpp}]
Since the circumcircle is mapped to the circumcircle, it is a contour
line of $|\mu_{f}|$ (Theorems~\ref{thm:hyppencil} and~\ref{thm:circles}). This proves the
first two equalities of equation~\eqref{eq:cpp_distortion}. It remains to show the last equality
involving the eccentricity of the affine map. We will avoid any
calculation of derivatives. 

Given three lines $\ell_1,\ell_2,\ell_3$ intersecting in one point $P$
and three lines $\tilde{\ell}_1, \tilde{\ell}_2, \tilde{\ell}_3$
intersecting in one point $\tilde P$, there exists an affine map $F$
mapping $\ell_j$ to $\tilde{\ell}_j$ ($j=1,2,3$), and this map is
unique up to post-composition with a homothety centered at $\tilde P$
(or, which amounts to the same, pre-composition with a homothety
centered at $P$).  The value of $\mu_F$ is therefore uniquely
determined by the lines.

The affine map $h$ maps the point $A$ to the point $\tilde A$ and the
lines $\ell_{1}=AB$, $\ell_{2}=AC$, and the line $\ell_{3}$ parallel
to $CB$ through $A$ to the lines $\tilde \ell_{1}=\tilde A\tilde B$,
$\tilde \ell_{2}=\tilde A\tilde C$, and the line $\tilde\ell_{3}$
parallel to $\tilde C\tilde B$ through $\tilde A$ (see
Figure~\ref{fig:qlinear}).
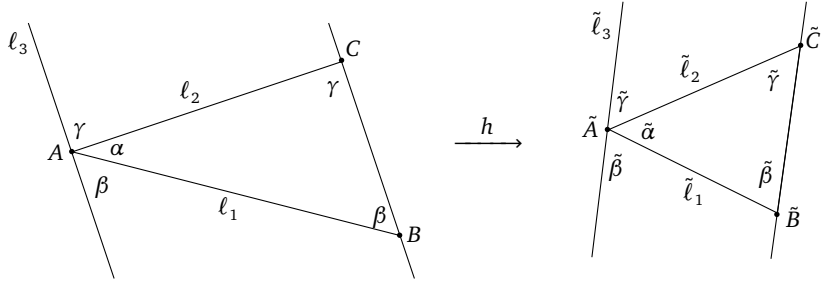
\begin{figure}
\labellist
\small\hair 2pt
 \pinlabel {$A$} by 1.7 0 at 33 95
 \pinlabel {$B$} by -1.2 -0.2 at 278 32
 \pinlabel {$C$} [bl] at 234 164
 \pinlabel {$\ell_{1}$} [t] at 150 65
 \pinlabel {$\ell_{2}$} [br] at 132 129
 \pinlabel {$\ell_{3}$} [r] at 5 177
 \pinlabel {$\tilde \ell_{1}$} [t] at 497 80
 \pinlabel {$\tilde \ell_{2}$} [br] at 505 145
 \pinlabel {$\tilde \ell_{3}$} [r] at 440 190
 \pinlabel {$\tilde A$} by 1.7 -0.3 at 432 112
 \pinlabel {$\tilde B$} by -1.5 0.2 at 559 48
 \pinlabel {$\tilde C$} by -1 -0.4 at 577 175
 \pinlabel {$\alpha$} [l] at 57 97
 \pinlabel {$\beta$} [ ] at 263 46
 \pinlabel {$\gamma$} [t] at 228 153
 \pinlabel {$\beta$} [tl] at 47 80
 \pinlabel {$\gamma$} [ ] at 39 111
 \pinlabel {$\tilde\alpha$} [l] at 453 111
 \pinlabel {$\tilde\beta$} [b] at 551 64
 \pinlabel {$\tilde\gamma$} [tr] at 564 160
 \pinlabel {$\tilde\beta$} [t] at 439 98
 \pinlabel {$\tilde\gamma$} [ ] at 444 133
 \pinlabel {$\xrightarrow{\displaystyle\quad h\quad}$} [ ] at 344 110
\endlabellist
\centering
\includegraphics[scale=0.5]{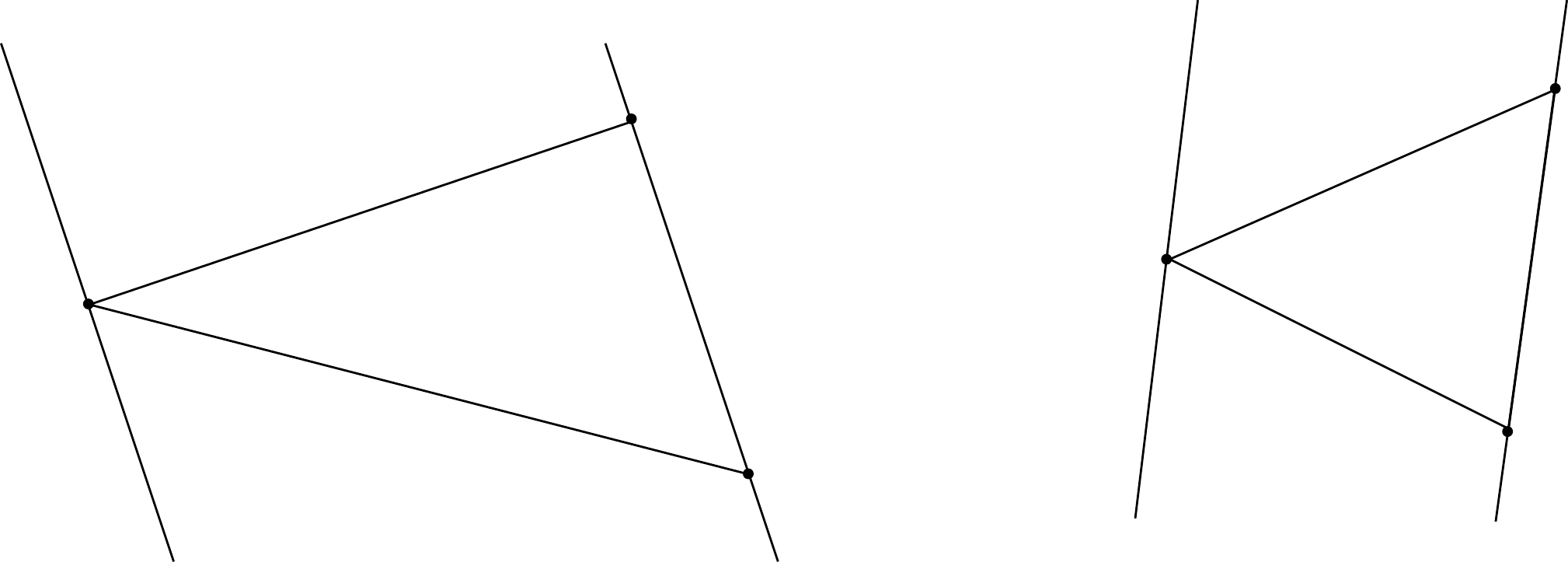}
\caption{The affine map $h$ from $T$ onto $\tilde T$. }
\label{fig:qlinear}
\end{figure}

The circumcircle preserving projective map $f$ also maps $\ell_{1}$ to
$\tilde\ell_{1}$ and $\ell_{2}$ to $\tilde\ell_{2}$, and it maps the
tangent $\ell_{3}'$ to the circumcircle at $A$ to the tangent
$\tilde\ell_{3}'$ to the circumcircle at $\tilde A$ (see
Figure~\ref{fig:qcircum}).
\begin{figure}
\labellist
\small\hair 2pt
 \pinlabel {$A$} by 1.7 0 at 25 157
 \pinlabel {$B$} by -1.2 0 at 270 94
 \pinlabel {$C$} [bl] at 226 225
 \pinlabel {$\tilde A$} by 1.7 -0.3 <2pt, -10pt> at 424 173
 \pinlabel {$\tilde B$} [tl] <2pt, -10pt> at 551 109
 \pinlabel {$\tilde C$} [bl] <2pt, -10pt> at 569 237
 \pinlabel {$\ell_{1}$} [t] at 142 126
 \pinlabel {$\ell_{2}$} [br] at 127 191
 \pinlabel {$\ell_{3}'$} [r] at 44 240
 \pinlabel {$\tilde \ell_{1}$} [t] <2pt, -10pt> at 490 141
 \pinlabel {$\tilde \ell_{2}$} [br] <2pt, -10pt> at 495 205
 \pinlabel {$\tilde \ell_{3}'$} [r] <2pt, -10pt> at 416 237
 \pinlabel {$\alpha$} [ ] at 60 158
 \pinlabel {$\beta$} [ ] at 251 112
 \pinlabel {$\gamma$} [t] at 220 213
 \pinlabel {$\beta$} [ ] at 48 179
 \pinlabel {$\gamma$} [ ] at 33 141
 \pinlabel {$\tilde\alpha$} [ ] <2pt, -10pt> at 449 173
 \pinlabel {$\tilde\beta$} [ ] <2pt, -10pt> at 543 129
 \pinlabel {$\tilde\gamma$} [ ] <2pt, -10pt> at 552 216
 \pinlabel {$\tilde\beta$} [ ] <2pt, -10pt> at 437 193
 \pinlabel {$\tilde\gamma$} [ ] <2pt, -10pt> at 441 153
 \pinlabel {$\xrightarrow{\displaystyle\quad f\quad}$} [ ] at 339 157
\endlabellist
\centering
\includegraphics[scale=0.5]{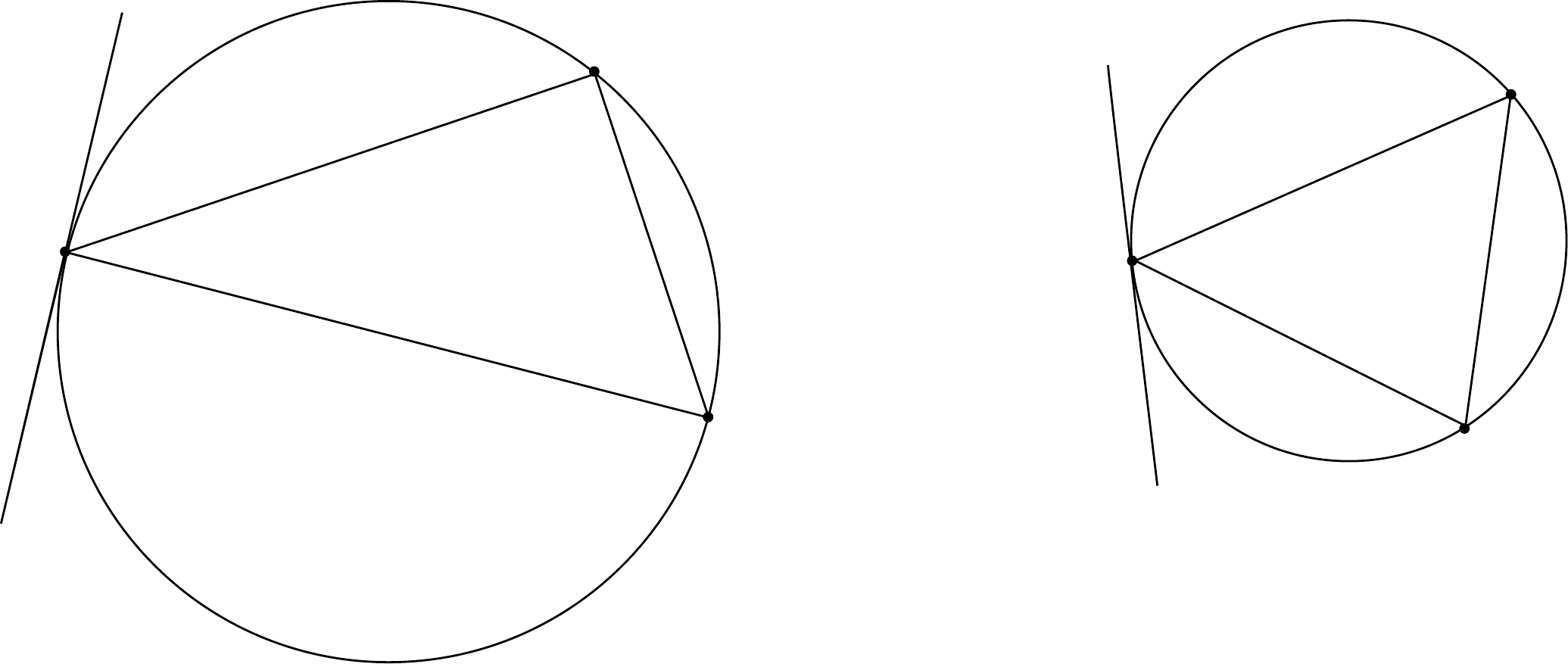}
\caption{The circumcircle preserving map $f$ from $T$ onto $\tilde T$.}
\label{fig:qcircum}
\end{figure}
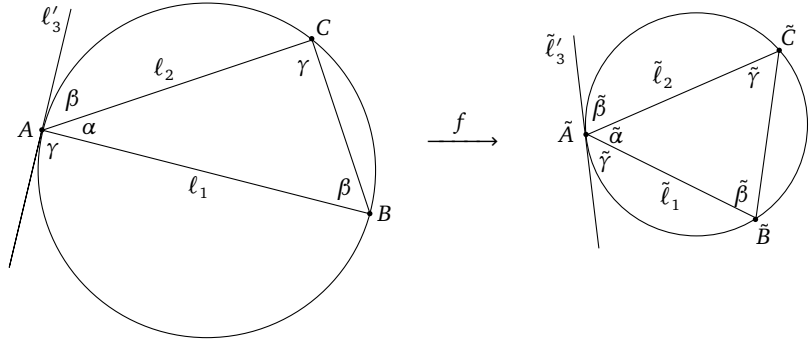
The affine approximation of $f$ at $A$ (that is, in affine coordinates,
the first order Taylor polynomial) also maps these lines to the same
lines. Considering the angles at~$A$ and~$\tilde A$ in
Figures~\ref{fig:qlinear} and~\ref{fig:qcircum}, one finds that the
affine approximation of $f$ at $A$ is equal to $h$ up to composition
with
similarity transformations. Thus, $|\mu_{f}(A)|=|\mu_{h}|$.
\end{proof}

\section{The angle bisector preserving projective map}
\label{sec:angle_bisector}

Which projective transformation between two given triangles minimizes
the maximal dilatation? Since the maximal dilatation is
attained at a vertex (Corollary~\ref{cor:max_at_vertex}), it is enough
to minimize the maximal dilatation at the vertices. As it turns out,
there is a unique projective transformation that simultaneously
minimizes the dilatation at all three vertices:

\begin{theorem}
  \label{thm:optimal}
  Of all projective maps that map a triangle $T\subset\C$ with
  vertices $A$, $B$, $C$ onto a triangle $\tilde T\subset\C$ with
  vertices $\tilde A$, $\tilde B$, $\tilde C$, which are labeled in
  the same orientation so that the maps are orientation preserving,
  the one that maps the angle bisectors to the angle bisectors
  simultaneously minimizes $|\mu(A)|$, $|\mu(B)|$, and
  $|\mu(C)|$. That is, the angle bisector preserving projective
  transformation $f$ satisfies
  \begin{equation*}
    \label{eq:simul_min}
    |\mu_{f}(A)|\leq |\mu_{g}(A)|,\quad
    |\mu_{f}(B)|\leq |\mu_{g}(B)|,\quad
    |\mu_{f}(C)|\leq |\mu_{g}(C)|
  \end{equation*}
  for all projective transformations $g$ with $g(A)=\tilde{A}$,
  $g(B)=\tilde{B}$, $g(C)=\tilde{C}$, $g(T)=\tilde T$.
\end{theorem}

This theorem follows immediately from the following Lemma, whose proof
we leave to the reader. (All arguments we know rely ultimately on one
or another more or less direct calculation.)

\begin{lemma}
\label{lem:anglebisector}
Of all linear maps in $\SL_{2}(\R)$ that map two one-dimensional
subspaces $\R v$, $\R w$ onto two one-dimensional subspaces
$\R\tilde v$, $\R\tilde w$, the map $f\in\SL_{2}(\R)$ has the least
dilatation if and only if it maps one, and hence both, of the angle bisectors
$\R\big(\frac{v}{|v|}\pm\frac{w}{|w|}\big)$ to the corresponding angle
bisector
$\R\big(\frac{\tilde v}{|\tilde v|}\pm\frac{\tilde w}{|\tilde
    w|}\big)$.
(This determines $f$ uniquely up to sign.)
\end{lemma}

\begin{figure}
  \centering
  \includegraphics{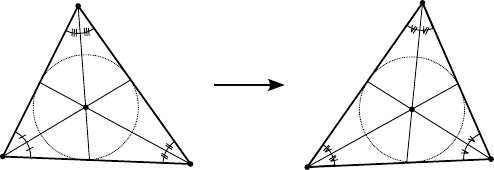}
  \caption{The angle bisector preserving projective transformation
    maps incircle center to incircle center, but it does in general
    not map incircle to incircle.}
  \label{fig:anglebisectors}
\end{figure}

The characterization of discrete conformal equivalence in terms of
projective maps (Definition and Theorem~\ref{def:dconf}~(iii)) remains true
if ``circumcircle preserving'' is replaced by ``angle bisector preserving'':

\begin{theorem}
  \label{thm:app_fit}
  The angle bisector preserving projective maps between corresponding
  triangles of combinatorially equivalent triangulations fit together
  continuously across edges if and only if the triangulations are
  discretely conformally equivalent.
\end{theorem}

\begin{proof}
  This follows immediately from the elementary angle bisector theorem
  (Figure~\ref{fig:anglebisectors2})
  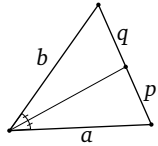
\begin{figure}
    \centering
    \labellist
    \small\hair 2pt
    \pinlabel {$a$} [tl] at 26 2
    \pinlabel {$b$} [br] at 17 24
    \pinlabel {$p$} [l] at 49 15
    \pinlabel {$q$} [l] at 39 36
    \endlabellist
    \includegraphics{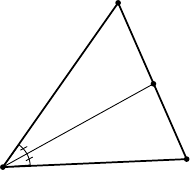}
    \caption{Angle bisector theorem: $\frac{a}{b}=\frac{p}{q}$.}
    \label{fig:anglebisectors2}
  \end{figure}
  and the characterization of discrete conformal equivalence in terms
  of length cross ratios (Definition and Theorem~\ref{def:dconf}~(ii)).
\end{proof}

\section{A one-parameter family of piecewise projective
  interpolations}
\label{sec:conclusion}

The angle bisector preserving projective transformation maps
incircle center to incircle center. The incircle center in
a triangle $ABC$ has barycentric coordinates $[a,b,c]$,
where $a$, $b$, $c$ denote the lengths of opposite sides.

The circumcircle preserving projective transformation maps symmedian
point to symmedian point (see Figure~\ref{fig:cpp}). The symmedian point
has barycentric coordinates $[a^{2},b^{2},c^{2}]$.
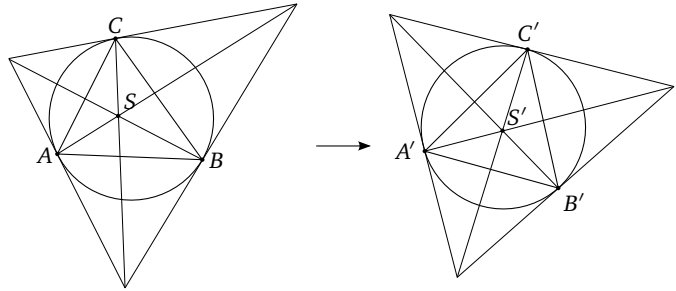
\begin{figure}
  \centering
  \labellist
  \small\hair 3pt
  \pinlabel {$A$} [r] at 19 50
  \pinlabel {$B$} [l] at 72 48
  \pinlabel {$C$} [b] at 40 93
  \pinlabel {$S$} [bl] at 41 65
  \pinlabel {$A'$} [r] at 155 52
  \pinlabel {$B'$} [tl] at 205 38
  \pinlabel {$C'$} [b] at 194 90
  \pinlabel {$S'$} [bl] at 184 59
  \endlabellist
  \includegraphics{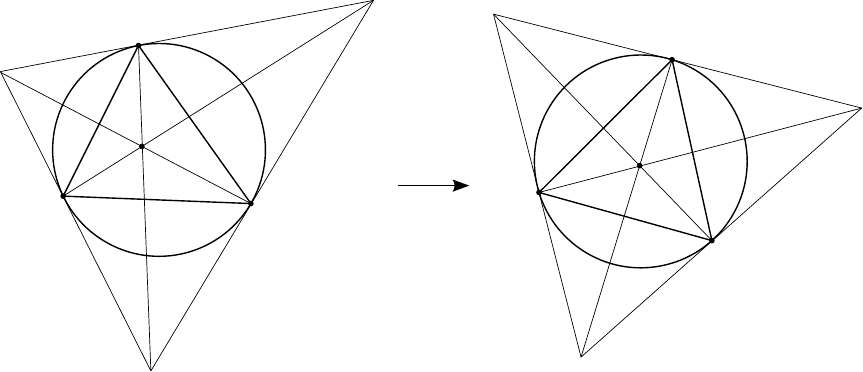}
  \caption{The circumcircle preserving projective transformation
    mapping $ABC$ to $A'B'C'$ maps the point $S$ to the point
    $S'$. The points $S$ and $S'$ are called the \emph{symmedian},
    \emph{Lemoine}, or \emph{Grebe point} of the triangles $ABC$ and
    $A'B'C'$, respectively. (They are the \emph{Gergonne points} of the
    triangles formed by the tangents.)}
  \label{fig:cpp}
\end{figure}

In the affine case, the barycenter is mapped to the barycenter. Its 
barycentric coordinates are $[1,1,1]$.

\begin{definition}
  For $t\in\R$, let the \emph{exponent-$t$-center} in a triangle $ABC$ be the
  point with barycentric coordinates $[a^{t},b^{t},c^{t}]$, where
  $a$, $b$, $c$ are the lengths of opposite sides.
\end{definition}

\begin{remark}
  For parameter values different from $t=0$ (barycenter), $t=1$
  (incircle center), and $t=2$ (symmedian point), the
  exponent-$t$-centers of a triangle seem to be fairly esoteric triangle
  centers. For example, the values $t=3,4$ correspond to triangle
  centers $X(31)$ and $X(32)$ in Kimberling's list~\cite{Kim}, and the
  values $t=-1,-2$ correspond to $X(75)$ and $X(76)$. If there are any
  other exponent-$t$-centers in Kimberling's list, they have indices greater
  than~300.
\end{remark}

The proof of Theorem~\ref{thm:app_fit} obviously generalizes to
``exponent-$t$-center preserving projective maps'':

\begin{theorem}
  For any $t\in\R$, the projective maps between corresponding
  triangles of combinatorially equivalent triangulations that map
  exponent-$t$-centers to exponent-$t$-centers fit together
  continuously across edges if and for $t\neq 0$ only if the
  triangulations are discretely conformally equivalent.
\end{theorem}

This leads to a one-parameter family of exponent-$t$-center preserving
piecewise projective interpolation schemes for discretely conformally
equivalent triangulations. The parameter value $t=0$ corresponds to
piecewise linear interpolation, $t=1$ to angle bisector preserving,
and $t=2$ to circumcircle preserving piecewise projective
interpolation.  Figure~\ref{fig:family} visualizes these interpolation
schemes for some values of $t$ using the same example as
Figure~\ref{fig:pl_vs_cpp}, middle.
\begin{figure}
  \newlength{\xwidth}
  \setlength{\xwidth}{0.44\textwidth}
  \newlength{\xheight}
  \settowidth{\xheight}{\includegraphics[height=\xwidth]{figures/pp_interpolate0_0}}
  \centering
  \vspace*{-0.1\xheight}
  \begin{tikzpicture}
    \node[anchor=south west, rotate=90] at (0\xwidth,0\xheight) {\includegraphics[width=\xwidth]{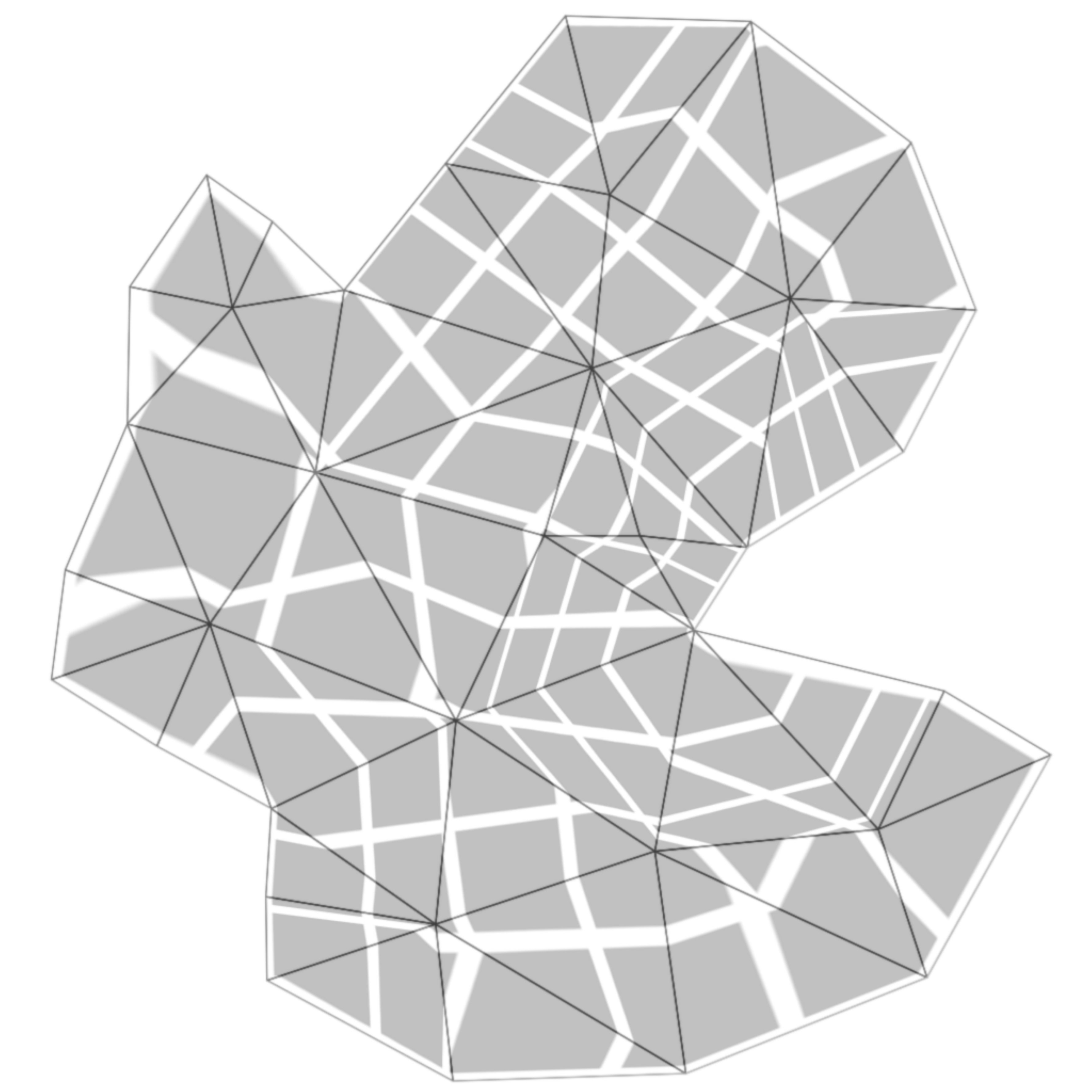}};
    \node[anchor=south west, rotate=90] at (0\xwidth,0.05\xheight) {$t=-1.0$};
    \node[anchor=south west, rotate=90] at (0\xwidth,0.96\xheight) {\includegraphics[width=\xwidth]{figures/pp_interpolate0_0}};
    \node[anchor=south west, rotate=90] at (0\xwidth,1.01\xheight) {$0.0$ (PL)};
    \node[anchor=south west, rotate=90] at (0\xwidth,1.92\xheight) {\includegraphics[width=\xwidth]{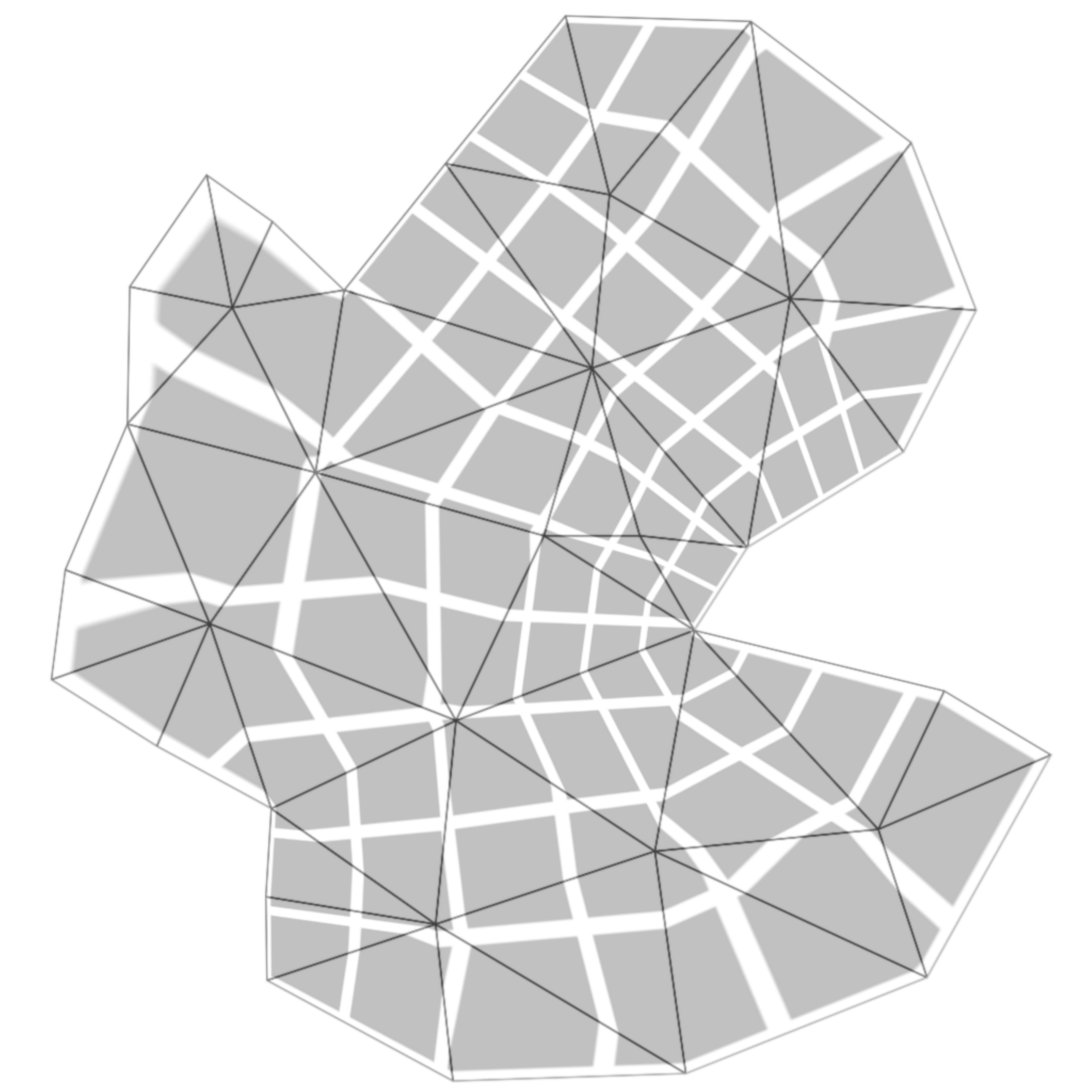}};
    \node[anchor=south west, rotate=90] at (0\xwidth,1.97\xheight) {$0.5$};
    \node[anchor=south west, rotate=90] at (0\xwidth,2.88\xheight) {\includegraphics[width=\xwidth]{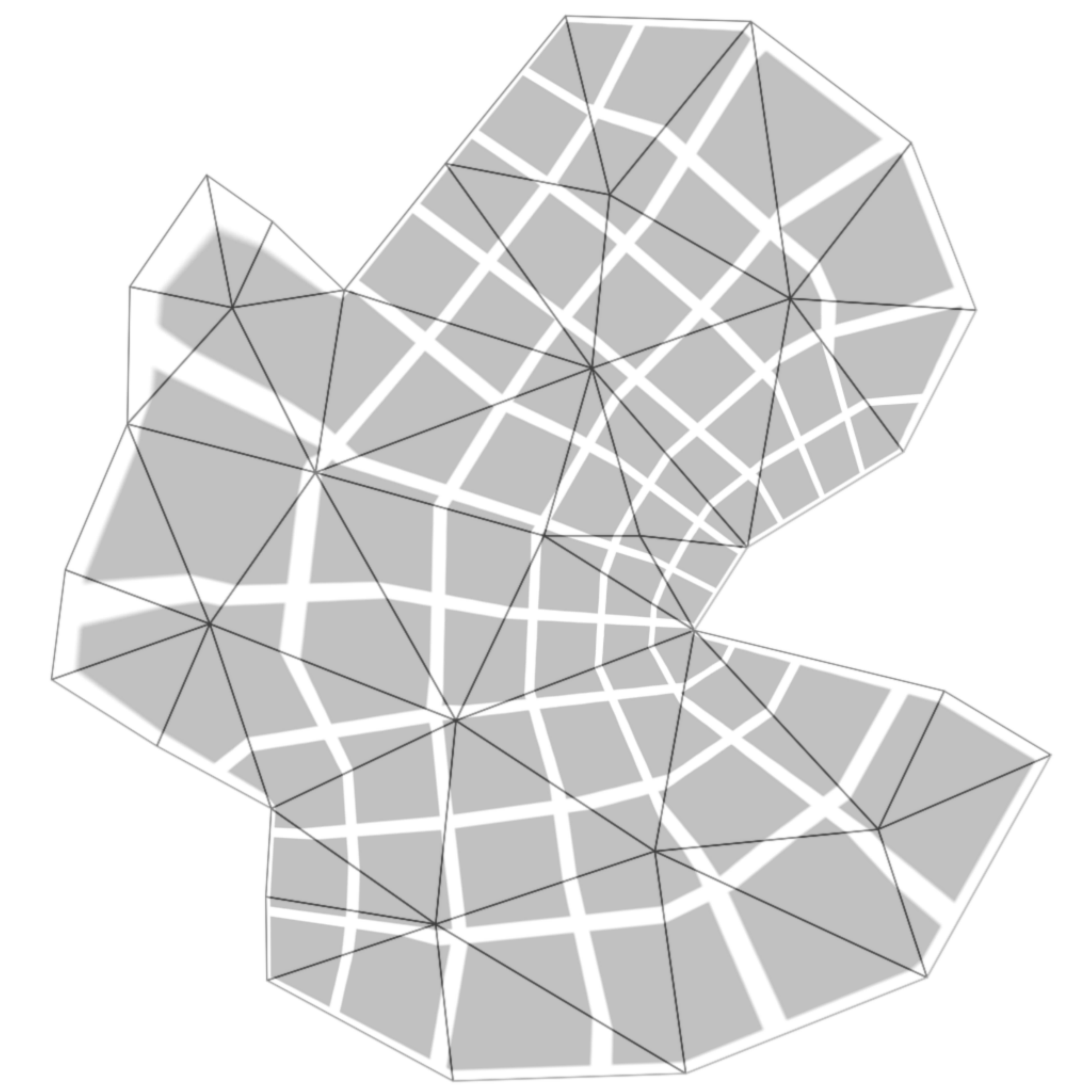}};
    \node[anchor=south west, rotate=90] at (0\xwidth,2.93\xheight) {$1.0$ (APP)};
    \node[anchor=south west, rotate=90] at (1\xwidth,2.88\xheight) {\includegraphics[width=\xwidth]{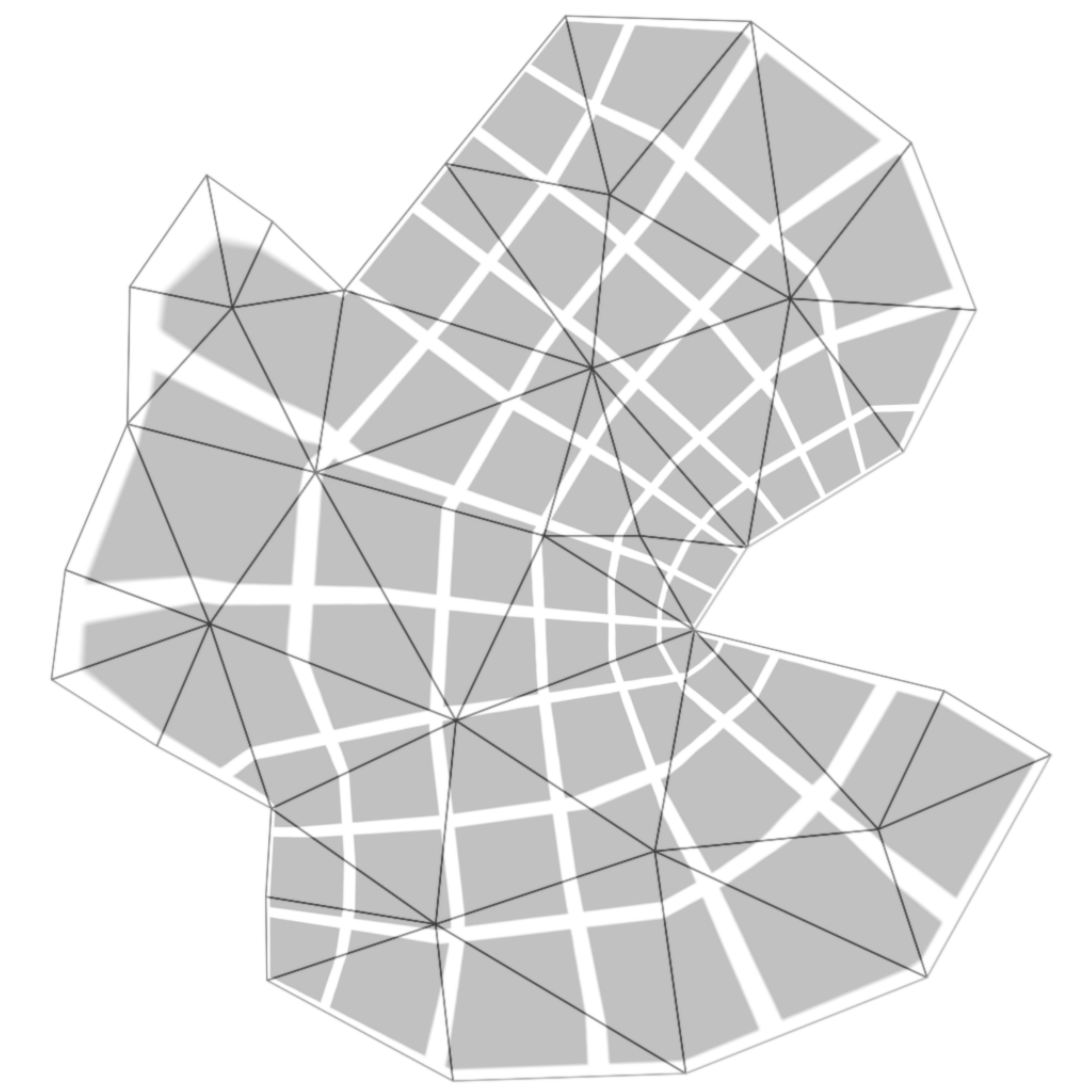}};
    \node[anchor=south west, rotate=90] at (1\xwidth,2.93\xheight) {$1.5$};
    \node[anchor=south west, rotate=90] at (1\xwidth,1.92\xheight) {\includegraphics[width=\xwidth]{figures/pp_interpolate2_0}};
    \node[anchor=south west, rotate=90] at (1\xwidth,1.97\xheight) {$2.0$ (CPP)};
    \node[anchor=south west, rotate=90] at (1\xwidth,0.96\xheight) {\includegraphics[width=\xwidth]{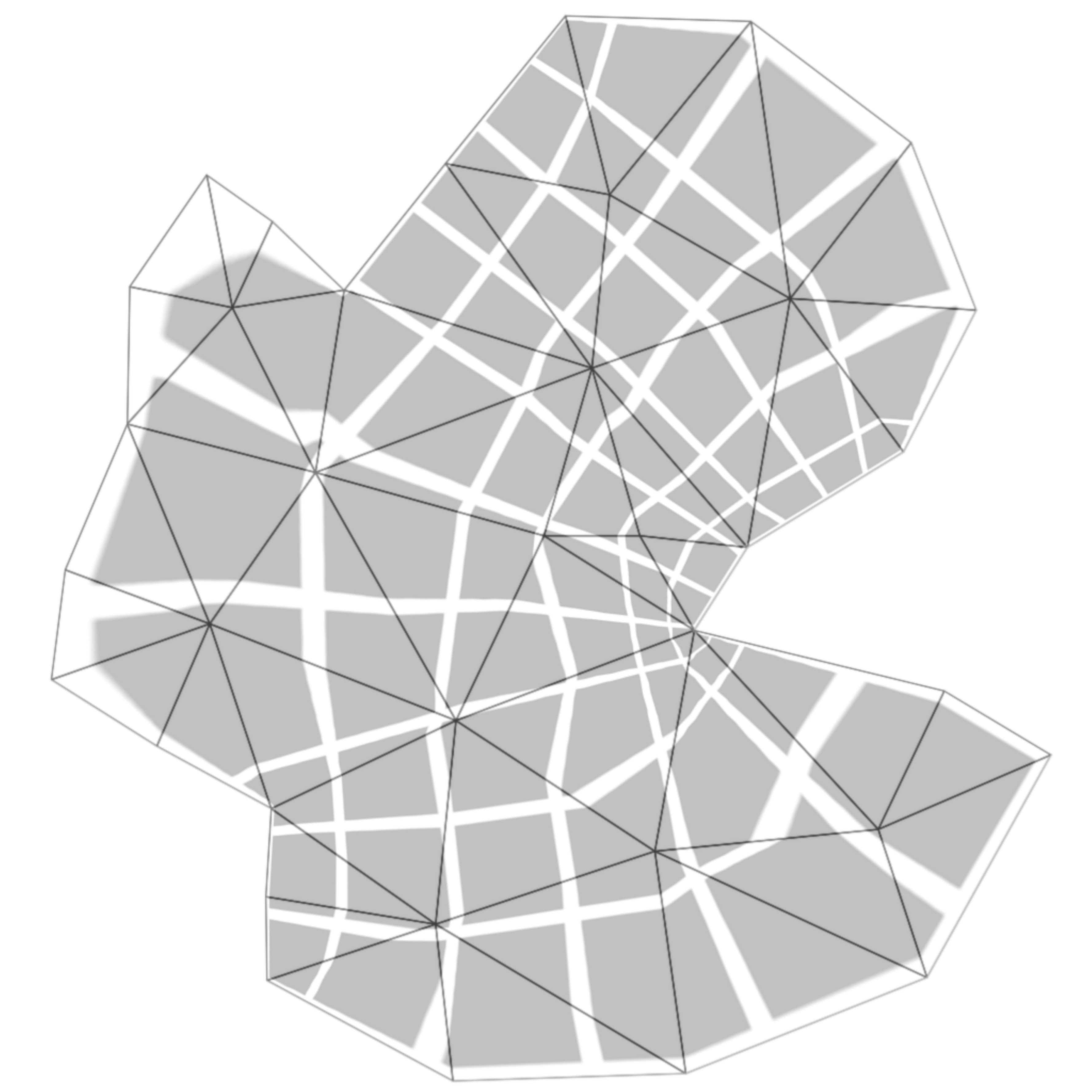}};
    \node[anchor=south west, rotate=90] at (1\xwidth,1.01\xheight) {$2.5$};
    \node[anchor=south west, rotate=90] at (1\xwidth,0\xheight) {\includegraphics[width=\xwidth]{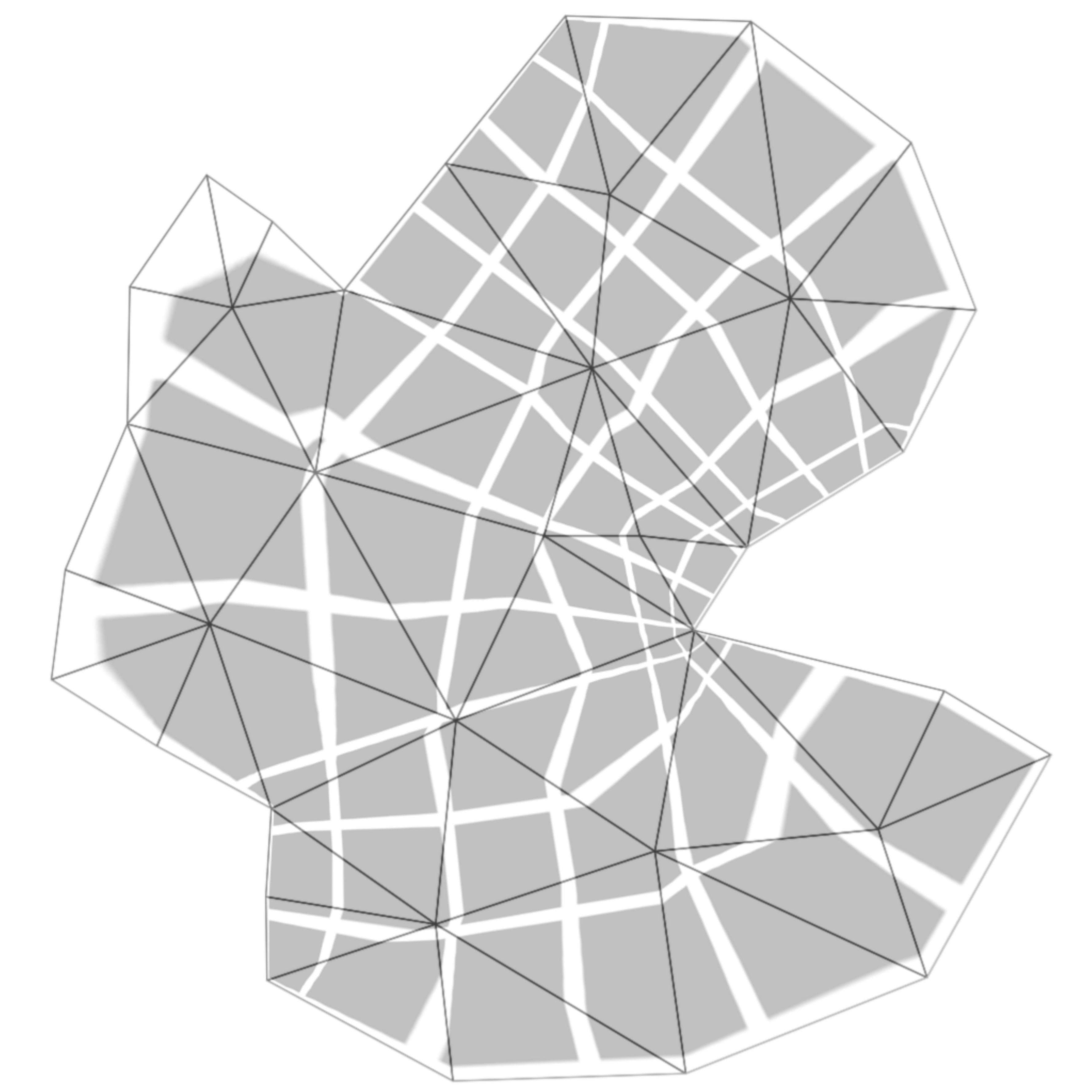}};
    \node[anchor=south west, rotate=90] at (1\xwidth,0.05\xheight) {$3.0$};
  \end{tikzpicture}
  \vspace*{-0.1\xheight}
  \caption{One-parameter family of piecewise projective
    interpolations.}
  \label{fig:family}
\end{figure}

\begin{remark}
  \label{rem:why}
  In our eyes, circumcircle preserving piecewise projective
  interpolation ($t=2$, CPP) looks slightly better than angle bisector
  preserving interpolation ($t=1$, APP). We have found this
  in other examples as well. Since APP has in general
  lower maximal dilatation per triangle, this indicates that low
  dilatation is not what makes CPP interpolation look
  better. We do not know which mathematical property of CPP
  interpolation accounts for the nicer appearance.
\end{remark}

\section*{Acknowledgement} This research was supported by the DFG~Collaborative
  Research Center TRR~109 ``Discretization in Geometry and Dynamics''.

\bibliographystyle{abbrv}
\bibliography{qcp}

\vspace{2\baselineskip}\noindent%
Stefan Born
\href{mailto:born@math.tu-berlin.de}{<born@math.tu-berlin.de>}\\
Ulrike B\"ucking
\href{mailto:buecking@math.tu-berlin.de}{<buecking@math.tu-berlin.de>}\\
Boris Springborn
\href{mailto:boris.springborn@tu-berlin.de}{<boris.springborn@tu-berlin.de>}

\medskip\noindent%
Technische Universit\"at Berlin\\
Sekretariat MA 8-3\\
Strasse des 17.\ Juni 136\\
10623 Berlin

\end{document}